\documentclass[sn-mathphys,Numbered]{sn-jnl}

\usepackage{graphicx}%
\usepackage{multirow}%
\usepackage{amsmath,amssymb,amsfonts}%
\usepackage{amsthm}%
\usepackage{mathrsfs}%
\usepackage[title]{appendix}%
\usepackage{xcolor}%
\usepackage{textcomp}%
\usepackage{manyfoot}%
\usepackage{booktabs}%
\usepackage{algorithm}%
\usepackage{algorithmicx}%
\usepackage{algpseudocode}%
\usepackage{listings}%

\usepackage[legacycolonsymbols]{mathtools}
\usepackage{physics}

\theoremstyle{thmstyleone}%
\newtheorem{thm}{Theorem}[section]
\theoremstyle{thmstyletwo}%

\theoremstyle{thmstylethree}%

\newtheorem{lem}{Lemma}[section]
\newtheorem{cor}{Corollary}[section]

\numberwithin{equation}{section}

\begin{document}

\title[Article Title]{Proof and generalization of conjectures of Ramanujan Machine}


\author*[1]{\fnm{Shuma} \sur{Yamamoto}}\email{yama.shu2006@gmail.com}

\affil*[1]{\orgname{Tokyo Metropolitan Nishi high school}, \city{Tokyo}, \country{Japan}}


\abstract{The Ramanujan Machine project predicts new continued fraction representations of numbers expressed by important mathematical constants. Generally, the value of a continued fraction is found by reducing it to a second order linear difference equation. In this paper, we prove 38 conjectures by solving the equation in two ways, use of a differential equation or application of Petkov\v{s}ek's algorithm. Especially, in the former way, we can get strong generalization of 31 conjectures.}



\maketitle

\section{Introduction}\label{sec1}
The Ramanujan Machine project\cite{RM0,RM1,RM2} predicts new continued fraction representations of numbers expressed by important mathematical constants; such as $\pi^2$, $\log(2)$, special values of the zeta function and Catalan's constant $G$. Let us represent continued fractions as 
$$
\text{CF}[a_n,b_n]=a_0+\cfrac{b_1}{a_1+\cfrac{b_2}{a_2+\ddots}}.
$$
In all the Ramanujan Machine's conjectures, $a_n$ and $b_n$ are polynomial functions of $n$ and satisfy $\deg(b_n)=2\deg(a_n)$. The conjectures are about the nontrivial value of the specific $\textup{CF}[a_n,b_n]$. The continued fractions of this type are closely related to proofs of the irrationality of mathematical constants. For instance, Ap\'{e}ry's proof of the irrationality of $\zeta(2)$ and $\zeta(3)$ essentially uses 
$$\frac{5}{\zeta(2)}=\textup{CF}\left[11n^2+11n+3,n^4\right]$$
and 
$$\frac{6}{\zeta(3)}=\textup{CF}\left[(2n+1)(17n^2+17n+5),-n^6\right]$$
respectively~\cite{AperyF,AperyE}. 
In order to explain our method, we first recall the following well-known fact.
\begin{lem}\label{l1.1}
Assume that $A_n$ and $B_n$ satisfy
\begin{align}
\left\{\begin{array}{l}
A_0=1,A_1=a_0,A_{n+1}=a_nA_n+b_nA_{n-1}, \\
B_0=0,B_1=1,\;\:B_{n+1}=a_nB_n+b_nB_{n-1}. \label{e1.1}
\end{array}
\right.
\end{align}
Then,

$$\textup{CF}[a_n,b_n]=\lim_{n\to\infty}\frac{A_n}{B_n}.$$
\end{lem}

Notice that $A_n$ and $B_n$ satisfy the same difference equation
\begin{equation}
    y_{n+1}=a_ny_n+b_ny_{n-1},\qquad n=1,2,3,\cdots. \label{e1.2}
\end{equation}
It is obvious that (\upshape\ref{e1.2}) can be easily solved when $\deg(a_n)=\deg(b_n)=0$. Actually, (\upshape\ref{e1.2}) can also be solved when $\deg(a_n)=1,\deg(b_n)=2$ because the exponential generating function $\sum_{n=0
}^\infty y_n\frac{x^n}{n!}$ of the solution of (\upshape\ref{e1.2}) satisfies a differential equation which comes down to the hypergeometric or confluent hypergeometric differential equation by a simple change of the independent variable $x$. In this paper, by considering the case $\deg(a_n)=2,\deg(b_n)=4$ in a similar way, we derive the following formula (Theorem~{\upshape\ref{t2.1}}) and prove Ramanujan Machine's 31 conjectures as its corollaries:
\begin{equation}
    \textup{CF}\left[\Delta\Bigl(n(n+\alpha)(n+\beta)\Bigr),-2n(n+\alpha)(n+\beta)(n+\gamma)\right]=4\left\{\sum_{n=0}^\infty \frac{\bigl(\frac{1}{2}\bigr)_n(\gamma+1)_n}{\bigl(\frac{\alpha+1}{2}\bigr)_{n+1}\bigl(\frac{\beta+1}{2}\bigr)_{n+1}}\right\}^{-1}, \label{e1.3}
\end{equation} where $\Delta$ is a difference operator defined by $\Delta f_n=f_{n+1}-f_n$. Furthermore, even if the series diverges, this formula holds due to its analytic cotinuation.(Lemma~{\upshape\ref{l2.3}}, Theorem~{\upshape\ref{t2.2}})

Alternatively, we can find the value of the continued fraction just by finding a special solution of (\upshape\ref{e1.2}) because the continued fraction can be expressed in the form of a series using $y_n$ (see Lemma~{\upshape\ref{l3.1}} below). We can sometimes obtain the special solution of (\upshape\ref{e1.2}) by applying Petkov\v{s}ek's algorithm~\cite{A=B}, which finds the solution of a given linear difference equation with polynomial coefficients under the assumption that the solution $y_n$ is hypergeometric, i.e. $y_n/y_{n-1}$ is a rational function of $n$. In fact, we can prove seven more conjectures by using this method. 

\section{Proof and Generalization of the conjectures for $\pi^2$, $\log(2)$ and Catalan's constant}
In what follows, we define $\sum_{k=0}^{-1}f_k=0$ and $\prod_{i=1}^0f_i=1$ for arbitrary sequence $f_n$ as convention. Moreover, in this section, $a_n$ and $b_n$ are denoted by $a(n)$ and $b(n)$ respectively, and we assume that both are polynomials satisfying $\deg(b)=2\deg(a)$.

In the case where $\deg(a)=1$, let $y_n$ be a sequence satisfying (\upshape\ref{e1.2}) and $y=y(x)$ be the generating function defined by
$$y=\sum_{n=0}^\infty y_n\frac{x^n}{n!}.$$
Then a differential equation that $y$ satisfies is given by using the differential operator $D_x=\frac{d}{dx}$ and the Euler operator $\vartheta_x=x\frac{d}{dx}$ as follows:
$$D_x^2y=a(\vartheta_x+1)D_xy+b(\vartheta_x+1)y.$$
In what follows, when the variable is obvious, the two operators will be denoted by $D$ and $\vartheta$. We can solve this equation because it is rewritten as
\begin{equation}
    P(x)\frac{d^2y}{dx^2}+Q(x)\frac{dy}{dx}+Ry=0, \label{e2.1}
\end{equation}
where $P(x)$ and $Q(x)$ are polynomials of degree two and one respectively, and transformed into the hypergeometric differential equation
$$x(1-x)\frac{d^2y}{dx^2}+(\gamma-(\alpha+\beta+1)x)\frac{dy}{dx}-\alpha\beta y=0$$
by variable translation $x\mapsto x+\delta$ and constant multiplication $x\mapsto cx$ in most cases.\footnote{If $P(x)$ has a multiple solution, (\upshape\ref{e2.1}) comes down to the confluent hypergeometric differential equation.} This equation is known to have the solution 
$$y=F(\alpha,\beta;\gamma;x)=\sum_{n=0}^\infty \frac{(\alpha)_n(\beta)_n}{(\gamma)_n}\frac{x^n}{n!}.$$
This can be derived from the fact that the equation is equivalent to the following:
$$\Bigl\{\left(\vartheta+\gamma\right)D-(\vartheta+\alpha)(\vartheta+\beta)\Bigr\}y=0.$$

Let us have the same argument when $\deg(a)=2$. Fix $s \not\in \mathbb{Z}_{\le 0}$ and put $y'_n = \frac{y_n}{(s)_n}$, where $y_n$ is a solution of (\upshape\ref{e1.2}). Define $y = y(x)$ by:
$$y=\sum_{n=0}^\infty y_n'\frac{x^n}{n!}.$$
Then $y_n'$ satisfies the following recurrence formula:
$$(n+s)(n+s-1)y_{n+1}'=(n+s-1)a(n)y_n'+b(n)y_{n-1}',\qquad n=1,2,3,\cdots.$$
Assume that $n+s-1$ divides $b(n)$. Then the formula is simplified as follows:
$$(n+s)y_{n+1}'=a(n)y_n'+c(n)y_{n-1}'\qquad\left(c(n)=\frac{b(n)}{n+s-1}\right).$$
Therefore, $y$ satisfies the differential equation
\begin{equation}
    \Bigl\{(\vartheta+s+1)D^2-a(\vartheta+1)D-c(\vartheta+1)\Bigr\}y=0. \label{e2.2}
\end{equation}

Before we discuss how to
solve this equation, let us first prove basic properties of the Euler operator.
\begin{lem}\label{l2.1}
Let $p(x)$ be a polynomial. Then, 
\begin{itemize}
    \item[\textup{1.}] $x^ap(\vartheta)=p(\vartheta-a)x^a\qquad(a\in\mathbb{C}).$
    \item[\textup{2.}] $x^nD^n=\vartheta^{\underline{n}}\qquad\left(n\in\mathbb{Z}_{\geq0},\vartheta^{\underline{n}}=\vartheta(\vartheta-1)\cdots(\vartheta-n+1)\right).$
    \item[\textup{3.}] $D^np(\vartheta)=p(\vartheta+n)D^n.$
\end{itemize}
\end{lem}
\begin{proof}
    \underline{1.} This formula holds for $a=1$ and $p(x)=x$. Namely, $x\vartheta$ and $(\vartheta-1)x$ are equal as operators. Thus, letting $q(x)$ be a polynomial and assuming that $q(x)=c(x-\alpha_1)(x-\alpha_2)\cdots(x-\alpha_n)$ we have
    \begin{align*}
    cx(\vartheta-\alpha_1)(\vartheta-\alpha_2)\cdots(\vartheta-\alpha_n)&=c(\vartheta-\alpha_1-1)x(\vartheta-\alpha_2)\cdots(\vartheta-\alpha_n)\\
    &=c(\vartheta-\alpha_1-1)(\vartheta-\alpha_2-1)x\cdots(\vartheta-\alpha_n)\\
    &\quad\vdots\\
    &=c(\vartheta-\alpha_1-1)(\vartheta-\alpha_2-1)\cdots(\vartheta-\alpha_n-1)x.
    \end{align*}
Therefore we get $xq(\vartheta)=q(\vartheta-1)x$. Here we transform the variable as $x=t^a$. From $\vartheta_x=\frac{1}{a}\vartheta_t$, we see that
    $$t^aq\left(\frac{\vartheta_t}{a}\right)=q\left(\frac{\vartheta_t}{a}-1\right)t^a.$$
    We obtain the formula by $t\mapsto x$ and $q(x)=p(ax)$.
    
    \underline{2.} We prove the formula by induction. It holds for $n=0$. Suppose that $x^{n-1}D^{n-1}=\vartheta^{\underline{n-1}}$. Then we have 
    $$x^nD^n=x(x^{n-1}D^{n-1})D=x\vartheta^{\underline{n-1}} D=(\vartheta-1)^{\underline{n-1}}xD=\vartheta^{\underline{n}},$$
    which completes the proof.

    \underline{3.} From the formula 1 and 2, we have
    $$D^np(\vartheta)=x^{-n}\vartheta^{\underline{n}}p(\vartheta)=x^{-n}p(\vartheta)\vartheta^{\underline{n}}=p(\vartheta+n)x^{-n}\vartheta^{\underline{n}}=p(\vartheta+n)D^n.$$
\end{proof}

As in the case where $\deg(a)=1$, we consider a variable translation in (\upshape\ref{e2.2}). To deal with the translation, the following lemma is useful.

\begin{lem}\label{l2.2}
Let $p(n)$ be a polynomial, and assume that $t=x-\delta$. Then,
$$p(\vartheta_x)=\sum_{k=0}^{\deg(p)} \frac{\Delta^kp(\vartheta_t)}{k!}(\delta D_t)^k.$$
\end{lem}
\begin{proof}
    Both sides of the formula are linear with respect to $p$. Thus we only have to prove the formula for a certain sequence of polynomials $p_m(n)$ satisfying $\deg(p_m)=m$.
    
    Let $p(n)=n^{\underline{m}}$. The following property is well-known:
    $$\Delta n^{\underline{m}}=m\cdot n^{\underline{m-1}}.$$
    The formula holds for $p(n)=n^{\underline{0}}=1$. Suppose that the formula holds for $p(n)=n^{\underline{m}}$. Then,
    \begin{eqnarray*}
        {\vartheta_x}^{\underline{m+1}}&=&\vartheta_x(\vartheta_x-1)^{\underline{m}}\\
        &=&(\vartheta_t+\delta D_t)\sum_{k=0}^m\binom{m}{k}(\vartheta_t-1)^{\underline{m-k}}(\delta D_t)^k\\
        &=&\sum_{k=0}^m\binom{m}{k}{\vartheta_t}^{\underline{m-k+1}}(\delta D_t)^k+\sum_{k=0}^m\binom{m}{k}{\vartheta_t}^{\underline{m-k}}(\delta D_t)^{k+1}\\
        &=&{\vartheta_t}^{\underline{m+1}}+(\delta D_t)^{m+1}+\sum_{k=1}^m\left\{\binom{m}{k}+\binom{m}{k-1}\right\}{\vartheta_t}^{\underline{m-k+1}}(\delta D_t)^k\\
        &=&\sum_{k=0}^{m+1}\binom{m+1}{k}{\vartheta_t}^{\underline{(m+1)-k}}(\delta D_t)^k.
    \end{eqnarray*}
    Thus the formula holds for $p(n)=n^{\underline{m+1}}$. Therefore it is proved inductively that the formula holds for $p(n)=n^{\underline{m}}$ for all non-negative integers $m$.
\end{proof}
Putting $t=x-\delta$ in (\upshape\ref{e2.2}) and applying Lemma~{\upshape\ref{l2.2}}, we have
\begin{eqnarray*}
    \left[\left(\vartheta_t+\delta D_t+s+1\right)D_t^2-\left\{a(\vartheta_t+1)+\Delta a(\vartheta_t+1)\cdot(\delta D_t)+\frac{\Delta^2a(\vartheta_t+1)}{2}(\delta D_t)^2\right\}D_t\right.\\
    \left.-\left\{c(\vartheta_t+1)+\Delta c(\vartheta_t+1)\cdot(\delta D_t)+\frac{\Delta^2c(\vartheta_t+1)}{2}(\delta D_t)^2+\frac{\Delta^3c(\vartheta_t+1)}{6}(\delta D_t)^3\right\}\right]y=0.
\end{eqnarray*}
Let $a^{(0)}$ and $b^{(0)}$ be the leading coefficients of $a(n)$ and $b(n)$ respectively. Then $\frac{\Delta^2a(\vartheta_t)}{2}=a^{(0)},\frac{\Delta^3c(\vartheta_t)}{6}=b^{(0)}$. Expanding the equation with respect to $D$, we obtain
\begin{eqnarray}\left[\delta(-b^{(0)}\delta^2-a^{(0)}\delta+1)D^3+\left\{-\frac{\delta^2}{2}\Delta^2c(\vartheta+1)-\delta\Delta a(\vartheta+1)+(\vartheta+s+1)\right\}D^2\right.\notag\\
+\left\{-\delta\Delta c(\vartheta+1)-a(\vartheta+1)\right\}D-c(\vartheta+1)\biggr]y(t)=0. \label{e2.3}
\end{eqnarray}

If two of the coefficients of $D^3$, $D^2$ and $D$ are zero, the equation is described more simply than (\upshape\ref{e2.2}). In particular, when we suppose that the coefficients of $D^3$ and $D$ equal zero, we obtain (\upshape\ref{e1.3}). 

\begin{thm}\label{t2.1}
    Assume that $\Re(2\gamma-\alpha-\beta)<1$ and that $\frac{\alpha+1}{2},\frac{\beta+1}{2},\gamma+1\not\in\mathbb{Z}_{\leq0}$. Then,
    $$\textup{CF}\left[\Delta\Bigl(n(n+\alpha)(n+\beta)\Bigr),-2n(n+\alpha)(n+\beta)(n+\gamma)\right]=4\left\{\sum_{n=0}^\infty \frac{\bigl(\frac{1}{2}\bigr)_n(\gamma+1)_n}{\bigl(\frac{\alpha+1}{2}\bigr)_{n+1}\bigl(\frac{\beta+1}{2}\bigr)_{n+1}}\right\}^{-1}.$$
\end{thm}
\begin{proof}
    Assume that $y_n$ satisfies the following recurrence formula:
    \begin{equation}
        y_{n+1}=\Delta\Bigl(n(n+\alpha)(n+\beta)\Bigr)y_n-2n(n+\alpha)(n+\beta)(n+\gamma)y_{n-1},\qquad n=1,2,3,\cdots, \label{e2.4}
    \end{equation}
    and let $y=y(x)$ be a function defined by
    $$y=\sum_{n=0}^\infty\frac{y_n}{(\gamma+1)_n}\frac{x^n}{n!}.$$
    By substituting $a(n)=\Delta(n(n+\alpha)(n+\beta))$, $s=\gamma+1$ and $c(n)=-2n(n+\alpha)(n+\beta)$ for (\upshape\ref{e2.2}), we see that $y$ satisfies the following differential equation:
    \begin{align}
    \Bigl\{(\vartheta_x+\gamma+2){D_x}^2-\Delta\Bigl((\vartheta_x+1)(\vartheta_x+\alpha+1)(\vartheta_x+\beta+1)\Bigr)D_x\notag\\
    +2(\vartheta_x+1)(\vartheta_x+\alpha+1)(\vartheta_x+\beta+1)\Bigr\}y=0. \label{e2.5}
    \end{align}
    With the translation of variable $t=x-\delta$, applying Lemma~{\upshape\ref{l2.2}} like (\upshape\ref{e2.3}), we have
    \begin{align*}
        \left[\delta(2\delta^2-3\delta+1){D_t}^3+\left\{\delta(\delta-1)\Delta^2\Bigl((\vartheta_t+1)(\vartheta_t+\alpha+1)(\vartheta_t+\beta+1)\Bigr)+(\vartheta_t+\gamma+2)\right\}{D_t}^2\right.\\
        \left.+(2\delta-1)\Delta\Bigl((\vartheta_t+1)(\vartheta_t+\alpha+1)(\vartheta_t+\beta+1)\Bigr)D_t+2(\vartheta_t+1)(\vartheta_t+\alpha+1)(\vartheta_t+\beta+1)\right]y=0.
    \end{align*}
    Putting $\delta=1/2$, we have
    $$\left\{\frac{\vartheta+(\alpha+\beta-2\gamma+2)}{4}{D}^2-(\vartheta+1)(\vartheta+\alpha+1)(\vartheta+\beta+1)\right\}y(t)=0.$$
    From Lemma~{\upshape\ref{l2.1}} we see that
    \begin{align*}
        \bigl(\vartheta+(\alpha+\beta-2\gamma+2)\bigr)D^2       &=\bigl(\vartheta+(\alpha+\beta-2\gamma+2)\bigr)t^{-2}\vartheta(\vartheta-1)\\
        &=(\vartheta+1)\bigl(\vartheta+(\alpha+\beta-2\gamma+2)\bigr)t^{-2}\vartheta.
    \end{align*}
    Substituting this, we have
    $$\left\{(\vartheta+1)\left(\vartheta+(\alpha+\beta-2\gamma+2)\right)\frac{1}{(2t)^2}\vartheta-(\vartheta+1)(\vartheta+\alpha+1)(\vartheta+\beta+1)\right\}y(t)=0.$$
    Thus,
    $$\left\{\left(\vartheta_t+(\alpha+\beta-2\gamma+2)\right)\frac{1}{(2t)^2}\vartheta_t-(\vartheta_t+\alpha+1)(\vartheta_t+\beta+1)\right\}y=\frac{C}{t},$$
    where $C$ is an arbitrary constant. Putting $u=(2t)^2$, we see from $\vartheta_t=2\vartheta_u$ that 
    $$\left\{\left(\vartheta_u+\left(\frac{\alpha+\beta}{2}-\gamma+1\right)\right)D_u-\left(\vartheta_u+\frac{\alpha+1}{2}\right)\left(\vartheta_u+\frac{\beta+1}{2}\right)\right\}y=\frac{C}{2\sqrt{u}}.$$
    Since $u=(2x-1)^2$, we see that $u=1$ at $x=0$. Hence, to solve (\upshape\ref{e2.4}) we should get power series solution at $u=1$ of this differential equation. Put $u=1-v$. Then, by applying Lemma~{\upshape\ref{l2.2}}, we obtain
    $$\left\{(\vartheta_v+\gamma+1)D_v-\left(\vartheta_v+\frac{\alpha+1}{2}\right)\left(\vartheta_v+\frac{\beta+1}{2}\right)\right\}y=\frac{C}{2}\sum_{n=0}^\infty \left(\frac{1}{2}\right)_n\frac{v^n}{n!}.$$
    Put $y(v)$ as follows:
    $$y=\sum_{n=0}^\infty z_n\frac{v^n}{n!}.$$
    Then,
    $$(n+\gamma+1)z_{n+1}-\left(n+\frac{\alpha+1}{2}\right)\left(n+\frac{\beta+1}{2}\right)z_n=\frac{C}{2}\left(\frac{1}{2}\right)_n.$$
    By solving this recurrence formula, by $\frac{\alpha+1}{2},\frac{\beta+1}{2},\gamma+1\not\in\mathbb{Z}_{\leq0}$, we get
    $$z_n=\frac{\bigl(\frac{\alpha+1}{2}\bigr)_n\bigl(\frac{\beta+1}{2}\bigr)_n}{(\gamma+1)_n}\left(z_0+\frac{C}{2}\sum_{k=0}^{n-1} \frac{\bigl(\frac{1}{2}\bigr)_k(\gamma+1)_k}{\bigl(\frac{\alpha+1}{2}\bigr)_{k+1}\bigl(\frac{\beta+1}{2}\bigr)_{k+1}}\right).$$
    Since $v=4x(1-x)$, we have
    $$\sum_{n=0}^\infty \frac{y_n}{(\gamma+1)_n}\frac{x^n}{n!}=\sum_{n=0}^\infty \frac{\bigl(\frac{\alpha+1}{2}\bigr)_n\bigl(\frac{\beta+1}{2}\bigr)_n}{(\gamma+1)_n}\left(z_0+\frac{C}{2}\sum_{k=0}^{n-1} \frac{\bigl(\frac{1}{2}\bigr)_k(\gamma+1)_k}{\bigl(\frac{\alpha+1}{2}\bigr)_{k+1}\bigl(\frac{\beta+1}{2}\bigr)_{k+1}}\right)\frac{\left\{4x(1-x)\right\}^n}{n!}.$$
    By the comparison of coefficients of degree 0 and 1, we have
    $$
    \left\{\begin{array}{l}
    y_0=z_0 ,\\
    \displaystyle\frac{y_1}{\gamma+1}=\frac{(\alpha+1)(\beta+1)}{\gamma+1}z_0+\frac{2C}{\gamma+1}.
    \end{array}
    \right.
    $$
    Therefore we have $z_0=y_0$ and $C=\frac{y_1-(\alpha+1)(\beta+1)y_0}{2}$.
    Thus, assuming that $A_n$ and $B_n$ are defined by the formula (\upshape\ref{e1.1}) with $a_n=\Delta\Bigl(n(n+\alpha)(n+\beta)\Bigr)$ and $b_n=-2n(n+\alpha)(n+\beta)(n+\gamma)$, we have
    \begin{align}
        \sum_{n=0}^\infty \frac{A_n}{(\gamma+1)_n}\frac{x^n}{n!}&=\sum_{n=0}^\infty \frac{\bigl(\frac{\alpha+1}{2}\bigr)_n\bigl(\frac{\beta+1}{2}\bigr)_n}{(\gamma+1)_n}\frac{\{4x(1-x)\}^n}{n!}\notag\\
        &=F\left(\frac{\alpha+1}{2},\frac{\beta+1}{2};\gamma+1;4x(1-x)\right) \label{e2.6}\\
        &=\sum_{n=0}^\infty\frac{\bigl(\frac{\alpha+1}{2}\bigr)_n\bigl(\frac{\beta+1}{2}\bigr)_n}{(\gamma+1)_nn!}4^n\sum_{i=0}^n\binom{n}{i}(-1)^ix^{n+i},\notag\\
        \sum_{n=0}^\infty \frac{B_n}{(\gamma+1)_n}\frac{x^n}{n!}&=\frac{1}{4}\sum_{n=0}^\infty \frac{\bigl(\frac{\alpha+1}{2}\bigr)_n\bigl(\frac{\beta+1}{2}\bigr)_n}{(\gamma+1)_n}\left(\sum_{k=0}^{n-1} \frac{\bigl(\frac{1}{2}\bigr)_k(\gamma+1)_k}{\bigl(\frac{\alpha+1}{2}\bigr)_{k+1}\bigl(\frac{\beta+1}{2}\bigr)_{k+1}}\right)\frac{\{4x(1-x)\}^n}{n!}\notag\\
        &=\frac{1}{4}\sum_{n=0}^\infty \frac{\bigl(\frac{\alpha+1}{2}\bigr)_n\bigl(\frac{\beta+1}{2}\bigr)_n}{(\gamma+1)_nn!}\left(\sum_{k=0}^{n-1} \frac{\bigl(\frac{1}{2}\bigr)_k(\gamma+1)_k}{\bigl(\frac{\alpha+1}{2}\bigr)_{k+1}\bigl(\frac{\beta+1}{2}\bigr)_{k+1}}\right)4^n\sum_{i=0}^n\binom{n}{i}(-1)^ix^{n+i}.\notag
    \end{align}
    Thus
    \begin{align*}
        A_n&=(\gamma+1)_nn!\sum_{k=\lceil\frac{n}{2}\rceil}^n\frac{\bigl(\frac{\alpha+1}{2}\bigr)_k\bigl(\frac{\beta+1}{2}\bigr)_k}{(\gamma+1)_kk!}4^k\binom{k}{n-k}(-1)^{n-k},\\
        B_n&=\frac{(\gamma+1)_nn!}{4}\sum_{k=\lceil\frac{n}{2}\rceil}^n\frac{\bigl(\frac{\alpha+1}{2}\bigr)_k\bigl(\frac{\beta+1}{2}\bigr)_k}{(\gamma+1)_kk!}4^k\binom{k}{n-k}(-1)^{n-k}\sum_{i=0}^{k-1} \frac{\bigl(\frac{1}{2}\bigr)_i(\gamma+1)_i}{\bigl(\frac{\alpha+1}{2}\bigr)_{i+1}\bigl(\frac{\beta+1}{2}\bigr)_{i+1}}.
    \end{align*}
    From
    $$\Gamma(z)\sim\sqrt{\frac{2\pi}{z}}\left(\frac{z}{e}\right)^z\quad(|\arg z|<\pi,z\rightarrow\infty),$$
    we see that
    $$\frac{(a)_n}{(b)_n}=\frac{\Gamma(b)}{\Gamma(a)}\cdot\frac{\Gamma(n+a)}{\Gamma(n+b)}\sim\frac{\Gamma(b)}{\Gamma(a)}\cdot e^{b-a}\sqrt{\frac{n+b}{n+a}}\left(1+\frac{a-b}{n+b}\right)^{n+b}(n+a)^{a-b}\sim\frac{\Gamma(b)}{\Gamma(a)}n^{a-b}=O(n^{\Re(a-b)})$$
    for $a,b\not\in\mathbb{Z}_{\leq0}$. Thus, when $\Re(2\gamma-\alpha-\beta)<1$, since 
    $$\frac{\bigl(\frac{1}{2}\bigr)_n(\gamma+1)_n}{\bigl(\frac{\alpha+1}{2}\bigr)_{n+1}\bigl(\frac{\beta+1}{2}\bigr)_{n+1}}=O\left(\frac{1}{n^{1+\epsilon}}\right)\qquad \left(\epsilon=\Re\left(\frac{\alpha+\beta+1}{2}-\gamma\right)>0\right)$$
    we have
    $$\sum_{i=0}^{k-1} \frac{\bigl(\frac{1}{2}\bigr)_i(\gamma+1)_i}{\bigl(\frac{\alpha+1}{2}\bigr)_{i+1}\bigl(\frac{\beta+1}{2}\bigr)_{i+1}}=\sum_{i=0}^{\infty} \frac{\bigl(\frac{1}{2}\bigr)_i(\gamma+1)_i}{\bigl(\frac{\alpha+1}{2}\bigr)_{i+1}\bigl(\frac{\beta+1}{2}\bigr)_{i+1}}+O(k^{-\epsilon})\quad(k\rightarrow\infty).$$
    Therefore it follows from Lemma~{\upshape\ref{l1.1}} that
    \begin{align*}
        &\textup{CF}\left[\Delta\Bigl(n(n+\alpha)(n+\beta)\Bigr),-2n(n+\alpha)(n+\beta)(n+\gamma)\right]=\lim_{n\rightarrow\infty}\frac{A_n}{B_n}\\
        &=\lim_{n\rightarrow\infty}\frac{\displaystyle\sum_{k=\lceil\frac{n}{2}\rceil}^n\frac{\bigl(\frac{\alpha+1}{2}\bigr)_k\bigl(\frac{\beta+1}{2}\bigr)_k}{(\gamma+1)_kk!}4^k\binom{k}{n-k}(-1)^{n-k}}{\displaystyle\frac{1}{4}\sum_{k=\lceil\frac{n}{2}\rceil}^n\frac{\bigl(\frac{\alpha+1}{2}\bigr)_k\bigl(\frac{\beta+1}{2}\bigr)_k}{(\gamma+1)_kk!}4^k\binom{k}{n-k}(-1)^{n-k}\left(\sum_{i=0}^{\infty} \frac{\bigl(\frac{1}{2}\bigr)_i(\gamma+1)_i}{\bigl(\frac{\alpha+1}{2}\bigr)_{i+1}\bigl(\frac{\beta+1}{2}\bigr)_{i+1}}+O(k^{-\epsilon})\right)}\\
        &=4\left\{\sum_{n=0}^\infty \frac{\bigl(\frac{1}{2}\bigr)_n(\gamma+1)_n}{\bigl(\frac{\alpha+1}{2}\bigr)_{n+1}\bigl(\frac{\beta+1}{2}\bigr)_{n+1}}\right\}^{-1}.
    \end{align*}
\end{proof}

To find the value of the continued fraction for $\Re(2\gamma-\alpha-\beta)\geq1$, we prove the following lemma.

\begin{lem}\label{l2.3}
    Define a function $H(\alpha,\beta;\gamma)$ of three complex variables  by
    $$H(\alpha,\beta;\gamma)=\frac{1}{\displaystyle\textup{CF}\left[\Delta\Bigl(n(n+\alpha)(n+\beta)\Bigr),-2n(n+\alpha)(n+\beta)(n+\gamma)\right]},$$
    where $\frac{\alpha+1}{2},\frac{\beta+1}{2},\gamma+1,\frac{\alpha+\beta+1}{2}-\gamma\not\in\mathbb{Z}_{\leq0}$.
    Then $H(\alpha,\beta;\gamma)$ is regular in each variable.
\end{lem}
\begin{proof}
    Put $A_n'=\frac{A_n}{(\gamma+1)_nn!}$, where $A_n$ is defined by (\upshape\ref{e1.1}). Since
    \begin{align*}
        A_{n+1}'&=\frac{\Delta\Bigl(n(n+\alpha)(n+\beta)\Bigr)}{(n+1)(n+\gamma+1)}A_n'-\frac{2(n+\alpha)(n+\beta)}{(n+1)(n+\gamma+1)}A_{n-1}'\\
        &=\left\{3+O\left(\frac{1}{n}\right)\right\}A_n'-\left\{2+O\left(\frac{1}{n}\right)\right\}A_{n-1}',
    \end{align*}
    we have
    \begin{equation}
        \lim_{n\rightarrow\infty}\frac{A_{n+1}'}{A_n'}\in\left\{\lambda|\lambda^2=3\lambda-2\right\}=\{1,2\}. \label{e2.7}
    \end{equation}
    
    Moreover, we get the following from Lemma~{\upshape\ref{l3.1}} discussed below:
    $$H(\alpha,\beta;\gamma)=\sum_{n=0}^\infty\frac{(\alpha+1)_n(\beta+1)_n}{(\gamma+1)_{n+1}(n+1)!}\frac{2^n}{A_n'A_{n+1}'}.$$
    When $\lim_{n\rightarrow\infty}A_{n+1}'/A_n'=2$, the ratio of the adjacent terms in the right-hand side converges to $\frac{1}{2}$. Thus the series converges absolutely uniformly in each variable. Therefore $H(\alpha,\beta;\gamma)$ is regular in each variable when $\lim_{n\rightarrow\infty}A_{n+1}'/A_n'=2$.
    
    By (\upshape\ref{e2.6}), we have
    $$\sum_{n=0}^\infty A_n'x^n=F\left(\frac{\alpha+1}{2},\frac{\beta+1}{2};\gamma+1;4x(1-x)\right).$$
    We can obtain the analytic continuation of $F(a,b;c;z)$ around $z=1$ by the following formulas\cite[pp.108-110]{EH}. Assume that $a,b,c\not\in\mathbb{Z}_{\leq0}$.
    \begin{enumerate}
        \item when $c-a-b$ is not an integer,
        \begin{align*}
        F(a,b;c;z)=&\frac{\Gamma(c)\Gamma(c-a-b)}{\Gamma(c-a)\Gamma(c-b)}F(a,b;a+b-c+1;1-z)\\&+(1-z)^{c-a-b}\frac{\Gamma(c)\Gamma(a+b-c)}{\Gamma(a)\Gamma(b)}F(c-a,c-b;c-a-b+1;1-z)
        \end{align*}
        \hfill$|\arg(1-z)|<\pi.$
        \item when $c-a-b=m$ is a non-negative integer,
        \begin{align*}
            &F(a,b;a+b+m;z)\\&=\frac{\Gamma(m)\Gamma(a+b+m)}{\Gamma(a+m)\Gamma(b+m)}\sum_{n=0}^{m-1}\frac{(a)_n(b)_n}{(1-m)_nn!}(1-z)^n\\
            &+(1-z)^m(-1)^m\frac{\Gamma(a+b+m)}{\Gamma(a)\Gamma(b)}\sum_{n=0}^\infty\frac{(a+m)_n(b+m)_n}{(n+m)!n!}\left\{h_n-\log(1-z)\right\}(1-z)^n
        \end{align*}
        \hfill $|\arg(1-z)|<\pi\;,\;h_n=\psi(n+1)+\psi(n+m+1)-\psi(a+n+m)-\psi(b+n+m)$.
        \item when $c-a-b=-m$ is a negative integer,
        \begin{align*}
            &F(a,b;a+b-m;z)\\&=(1-z)^{-m}\frac{\Gamma(m)\Gamma(a+b-m)}{\Gamma(a)\Gamma(b)}\sum_{n=0}^{m-1}\frac{(a-m)_n(b-m)_n}{(1-m)_nn!}(1-z)^n\\
            &+(-1)^m\frac{\Gamma(a+b-m)}{\Gamma(a-m)\Gamma(b-m)}\sum_{n=0}^\infty\frac{(a)_n(b)_n}{(n+m)!n!}\left\{h_n'-\log(1-z)\right\}(1-z)^n
        \end{align*}
        \hfill$|\arg(1-z)|<\pi\;,\;h_n'=\psi(n+1)+\psi(n+m+1)-\psi(a+n)-\psi(b+n)$.
    \end{enumerate}
    By substituting $a=\frac{\alpha+1}{2},b=\frac{\beta+1}{2},c=\gamma+1$ and $z=4x(1-x)$, we find that the analytic continuation of $\sum_{n=0}^\infty A_n'x^n$ has a branch point or a pole at $x=\frac{1}{2}$ when $\frac{\alpha+1}{2},\frac{\beta+1}{2},\gamma+1,\frac{\alpha+\beta+1}{2}-\gamma\not\in\mathbb{Z}_{\leq0}$. Therefore the radius of convergence of $\sum_{n=0}^\infty A_n'x^n$ is $\frac{1}{2}$ or less, hence we see from (\upshape\ref{e2.7}) that $\lim_{n\rightarrow\infty}A_{n+1}'/A_n'=2$. This completes the proof.
\end{proof}

\begin{thm}\label{t2.2} Assume that $\frac{\alpha+1}{2},\frac{\beta+1}{2},\gamma+1,\frac{\alpha+\beta+1}{2}-\gamma\not\in\mathbb{Z}_{\leq0}$. Then,
    $$\alpha\left(\alpha-2\gamma-1\right)H(\alpha,\beta;\gamma)-(\alpha-1)\left(\alpha+\beta-2\gamma-1\right)H(\alpha-2,\beta;\gamma)+1=0,$$
    $$2\gamma\left(2\gamma-\alpha-\beta-1\right)H(\alpha,\beta;\gamma)-\left(2\gamma-\alpha-1\right)\left(2\gamma-\beta-1\right)H(\alpha,\beta;\gamma-1)+1=0.$$
\end{thm}
\begin{proof}
    From Lemma~{\upshape\ref{l2.3}}, we see that $H(\alpha,\beta;\gamma)$ is regular in each variable because $\frac{\alpha+1}{2},\frac{\beta+1}{2},\gamma+1,\frac{\alpha+\beta+1}{2}-\gamma\not\in\mathbb{Z}_{\leq0}$. Therefore by the identity theorem, we only have to prove the formulas in the range of $\alpha,\beta$ and $\gamma$ which we can apply Theorem~{\upshape\ref{t2.1}}, i.e. the following holds:
    $$H(\alpha,\beta;\gamma)=\frac{1}{4}\sum_{n=0}^\infty\frac{\bigl(\frac{1}{2}\bigr)_n(\gamma+1)_n}{\bigl(\frac{\alpha+1}{2}\bigr)_{n+1}\bigl(\frac{\beta+1}{2}\bigr)_{n+1}}.$$
    
    We assume that $\Re(2\gamma-\alpha-\beta)<-1$ and that $\frac{\alpha-1}{2},\frac{\beta+1}{2},\gamma+1\not\in\mathbb{Z}_{\leq0}$ and we show the first formula as follows:
    \begin{align*}
        &\alpha\left(\alpha-2\gamma-1\right)H(\alpha,\beta;\gamma)-(\alpha-1)\left(\alpha+\beta-2\gamma-1\right)H(\alpha-2,\beta;\gamma)\\
        &=\sum_{n=0}^\infty\frac{\alpha}{2}\left(\frac{\alpha-1}{2}-\gamma\right)\frac{\bigl(\frac{1}{2}\bigr)_n(\gamma+1)_n}{\bigl(\frac{\alpha+1}{2}\bigr)_{n+1}\bigl(\frac{\beta+1}{2}\bigr)_{n+1}}-\frac{\alpha-1}{2}\left(\frac{\alpha+\beta-1}{2}-\gamma\right)\frac{\bigl(\frac{1}{2}\bigr)_n(\gamma+1)_n}{\bigl(\frac{\alpha-1}{2}\bigr)_{n+1}\bigl(\frac{\beta+1}{2}\bigr)_{n+1}}\\
        &=\sum_{n=0}^\infty\frac{\bigl(\frac{1}{2}\bigr)_n(\gamma+1)_n}{\bigl(\frac{\alpha+1}{2}\bigr)_n\bigl(\frac{\beta+1}{2}\bigr)_n}\left\{\frac{\frac{\alpha}{2}\Bigl(\frac{\alpha-1}{2}-\gamma\Bigr)}{\Bigl(n+\frac{\alpha+1}{2}\Bigr)\left(n+\frac{\beta+1}{2}\right)}-\frac{\frac{\alpha+\beta-1}{2}-\gamma}{n+\frac{\beta+1}{2}}\right\}\\
        &=\sum_{n=0}^\infty\frac{\bigl(\frac{1}{2}\bigr)_n(\gamma+1)_n}{\bigl(\frac{\alpha+1}{2}\bigr)_n\bigl(\frac{\beta+1}{2}\bigr)_n}\left\{\frac{\left(n+\frac{1}{2}\right)(n+\gamma+1)}{\Bigl(n+\frac{\alpha+1}{2}\Bigr)\left(n+\frac{\beta+1}{2}\right)}-1\right\}\\
        &=\sum_{n=0}^\infty\frac{\bigl(\frac{1}{2}\bigr)_{n+1}(\gamma+1)_{n+1}}{\bigl(\frac{\alpha+1}{2}\bigr)_{n+1}\bigl(\frac{\beta+1}{2}\bigr)_{n+1}}-\frac{\bigl(\frac{1}{2}\bigr)_n(\gamma+1)_n}{\bigl(\frac{\alpha+1}{2}\bigr)_n\bigl(\frac{\beta+1}{2}\bigr)_n}=-1.
    \end{align*}

    Next we assume that $\Re(2\gamma-\alpha-\beta)<1$ and that $\frac{\alpha+1}{2},\frac{\beta+1}{2},\gamma\not\in\mathbb{Z}_{\leq0}$ and we show the second formula as follows:
    \begin{align*}
        &2\gamma\left(2\gamma-\alpha-\beta-1\right)H(\alpha,\beta;\gamma)-\left(2\gamma-\alpha-1\right)\left(2\gamma-\beta-1\right)H(\alpha,\beta;\gamma-1)\\
        &=\sum_{n=0}^\infty\gamma\left(\gamma-\frac{\alpha+\beta+1}{2}\right)\frac{\bigl(\frac{1}{2}\bigr)_n(\gamma+1)_n}{\bigl(\frac{\alpha+1}{2}\bigr)_{n+1}\bigl(\frac{\beta+1}{2}\bigr)_{n+1}}-\left(\gamma-\frac{\alpha+1}{2}\right)\left(\gamma-\frac{\beta+1}{2}\right)\frac{\bigl(\frac{1}{2}\bigr)_n(\gamma)_n}{\bigl(\frac{\alpha+1}{2}\bigr)_{n+1}\bigl(\frac{\beta+1}{2}\bigr)_{n+1}}\\
        &=\sum_{n=0}^\infty\frac{\bigl(\frac{1}{2}\bigr)_n(\gamma)_n}{\bigl(\frac{\alpha+1}{2}\bigr)_n\bigl(\frac{\beta+1}{2}\bigr)_n}\left\{\frac{\left(\gamma-\frac{\alpha+\beta+1}{2}\right)(n+\gamma)}{\Bigl(n+\frac{\alpha+1}{2}\Bigr)\left(n+\frac{\beta+1}{2}\right)}-\frac{\Bigl(\gamma-\frac{\alpha+1}{2}\Bigr)\left(\gamma-\frac{\beta+1}{2}\right)}{\Bigl(n+\frac{\alpha+1}{2}\Bigr)\left(n+\frac{\beta+1}{2}\right)}\right\}\\
        &=\sum_{n=0}^\infty\frac{\bigl(\frac{1}{2}\bigr)_n(\gamma)_n}{\bigl(\frac{\alpha+1}{2}\bigr)_n\bigl(\frac{\beta+1}{2}\bigr)_n}\left\{\frac{\left(n+\frac{1}{2}\right)(n+\gamma)}{\Bigl(n+\frac{\alpha+1}{2}\Bigr)\left(n+\frac{\beta+1}{2}\right)}-1\right\}\\
        &=\sum_{n=0}^\infty\frac{\bigl(\frac{1}{2}\bigr)_{n+1}(\gamma)_{n+1}}{\bigl(\frac{\alpha+1}{2}\bigr)_{n+1}\bigl(\frac{\beta+1}{2}\bigr)_{n+1}}-\frac{\bigl(\frac{1}{2}\bigr)_n(\gamma)_n}{\bigl(\frac{\alpha+1}{2}\bigr)_n\bigl(\frac{\beta+1}{2}\bigr)_n}=-1.
    \end{align*}
\end{proof}

Many of the Ramanujan Machine's conjectures can be obtained as corollaries of Theorems~{\upshape\ref{t2.1}} and~{\upshape\ref{t2.2}}. The following are corollaries involving Catalan's constant $G$.
 
\begin{cor}\label{c1}
    $$\frac{1}{2G}=\textup{CF}\left[3n^2+3n+1,-2n^4\right]$$
\end{cor}
\begin{cor}\label{c2}
    $$\frac{2}{2G-1}=\textup{CF}\left[3n^2+3n+1,-2n^3(n+1)\right]$$
\end{cor}
\begin{cor}\label{c3}
    $$\frac{24}{18G-11}=\textup{CF}[3n^2+3n+1,-2n^3(n+2)]$$
\end{cor}
\begin{cor}\label{c4}
    $$\frac{720}{450G-299}=\textup{CF}[3n^2+3n+1,-2n^3(n+3)]$$
\end{cor}
\begin{cor}\label{c5}
    $$\frac{2}{2G-1}=\textup{CF}[3n^2+7n+3,-2n^3(n+2)]$$
\end{cor}
\begin{cor}\label{c6}
    $$\frac{4}{2G+1}=\textup{CF}[3n^2+7n+3,-2n^2(n+1)(n+2)]$$
\end{cor}
\begin{cor}\label{c7}
    $$\frac{16}{6G-1}=\textup{CF}[3n^2+7n+3,-2n^2(n+2)^2]$$
\end{cor}
\begin{cor}\label{c8}
    $$\frac{288}{90G-31}=\textup{CF}[3n^2+7n+3,-2n^2(n+2)(n+3)]$$
\end{cor}
\begin{cor}\label{c9}
    $$\frac{1}{2-2G}=\textup{CF}[3n^2+9n+7,-2n(n+1)^3]$$
\end{cor}
\begin{cor}\label{c10}
    $$\frac{24}{18G-11}=\textup{CF}[3n^2+11n+5,-2n^3(n+4)]$$
\end{cor}
\begin{cor}\label{c11}
    $$\frac{16}{6G-1}=\textup{CF}[3n^2+11n+5,-2n^2(n+1)(n+4)]$$
\end{cor}
\begin{cor}\label{c12}
    $$\frac{64}{18G+13}=\textup{CF}[3n^2+11n+5,-2n^2(n+2)(n+4)]$$
\end{cor}
\begin{cor}\label{c13}
    $$\frac{4}{6G-5}=\textup{CF}[3n^2+11n+9,-2n^2(n+2)^2]$$
\end{cor}
\begin{cor}\label{c14}
    $$\frac{8}{3-2G}=\textup{CF}[3n^2+11n+9,-2n(n+1)(n+2)^2]$$
\end{cor}
\begin{cor}\label{c15}
    $$\frac{32}{2G+5}=\textup{CF}[3n^2+11n+9,-2n(n+2)^3]$$
\end{cor}
\begin{cor}\label{c16}
    $$\frac{192}{18G+13}=\textup{CF}[3n^2+11n+9,-2n(n+2)^2(n+3)]$$
\end{cor}
\begin{cor}\label{c17}
    $$\frac{6}{17-18G}=\textup{CF}[3n^2+13n+13,-2n(n+1)^2(n+3)]$$
\end{cor}
\begin{cor}\label{c18}
    $$\frac{48}{90G-79}=\textup{CF}[3n^2+15n+15,-2n^2(n+2)(n+4)]$$
\end{cor}
\begin{cor}\label{c19}
    $$\frac{32}{19-18G}=\textup{CF}[3n^2+15n+15,-2n(n+1)(n+2)(n+4)]$$
\end{cor}
\begin{cor}\label{c20}
    $$\frac{128}{17-6G}=\textup{CF}[3n^2+15n+15,-2n(n+2)^2(n+4)]$$
\end{cor}
\begin{cor}\label{c21}
    $$\frac{8}{54G-49}=\textup{CF}[3n^2+15n+19,-2n(n+2)^3]$$
\end{cor}
\begin{cor}\label{c22}
    $$\frac{12}{83-90G}=\textup{CF}[3n^2+17n+23,-2n(n+1)(n+3)^2]$$
\end{cor}
\begin{proof}[Proof of Corollary~{\upshape\ref{c1}}]
    By substituting $\alpha=\beta=\gamma=0$ for Theorem~{\upshape\ref{t2.1}}, we have
    \begin{align*}
        \textup{CF}[3n^2+3n+1,-2n^4]&=\frac{1}{H(0,0;0)}=4\left\{\sum_{n=0}^\infty\frac{\left(\frac{1}{2}\right)_n(1)_n}{\bigl(\frac{1}{2}\bigr)_{n+1}\bigl(\frac{1}{2}\bigr)_{n+1}}\right\}^{-1}\\
        &=\left\{\sum_{n=0}^\infty\frac{(n!)^2}{(2n)!}\frac{4^{n}}{(2n+1)^2}\right\}^{-1}\\
        &=\frac{1}{2G},
    \end{align*}
    where we use \cite[Eq.3.114]{GS}.
\end{proof}
\begin{proof}[Proof of Corollary~{\upshape\ref{c2}}]
    By Theorem~{\upshape\ref{t2.2}} and Corollary~{\upshape\ref{c1}}, we have
    \begin{align*}
        \textup{CF}[3n^2+3n+1,-2n^3(n+1)]=\frac{1}{H(0,0;1)}=\frac{1}{\frac{H(0,0;0)-1}{2}}=\frac{2}{2G-1}.
    \end{align*}
\end{proof}
\begin{proof}[Proof of Corollary~{\upshape\ref{c3}}]
    By Theorem~{\upshape\ref{t2.2}} and Corollary~{\upshape\ref{c2}}, we have
    \begin{align*}
        \textup{CF}[3n^2+3n+1,-2n^3(n+2)]=\frac{1}{H(0,0;2)}=\frac{1}{\frac{9H(0,0;1)-1}{12}}=\frac{24}{18G-11}.
    \end{align*}
\end{proof}
\begin{proof}[Proof of Corollary~{\upshape\ref{c4}}]
    By Theorem~{\upshape\ref{t2.2}} and Corollary~{\upshape\ref{c3}}, we have
    \begin{align*}
        \textup{CF}[3n^2+3n+1,-2n^3(n+3)]=\frac{1}{H(0,0;3)}=\frac{1}{\frac{25H(0,0;2)-1}{30}}=\frac{720}{450G-299}.
    \end{align*}
\end{proof}
\begin{proof}[Proof of Corollary~{\upshape\ref{c5}}]
    By Theorem~{\upshape\ref{t2.2}} and Corollary~{\upshape\ref{c1}}, we have
    \begin{align*}
        \textup{CF}[3n^2+7n+3,-2n^3(n+2)]=\frac{1}{H(2,0;0)}=\frac{1}{\frac{H(0,0;0)-1}{2}}=\frac{2}{2G-1}.
    \end{align*}
\end{proof}
\begin{proof}[Proof of Corollary~{\upshape\ref{c6}}]
    By Theorem~{\upshape\ref{t2.2}} and Corollary~{\upshape\ref{c5}}, we have
    \begin{align*}
        \textup{CF}[3n^2+7n+3,-2n^2(n+1)(n+2)]=\frac{1}{H(2,0;1)}=\frac{1}{\frac{-H(2,0;0)-1}{-2}}=\frac{16}{6G-1}.
    \end{align*}
\end{proof}
\begin{proof}[Proof of Corollary~{\upshape\ref{c7}}]
    By Theorem~{\upshape\ref{t2.2}} and Corollary~{\upshape\ref{c6}}, we have
    \begin{align*}
        \textup{CF}[3n^2+7n+3,-2n^2(n+2)^2]=\frac{1}{H(2,0;2)}=\frac{1}{\frac{3H(2,0;1)-1}{4}}=\frac{4}{2G+1}.
    \end{align*}
\end{proof}
\begin{proof}[Proof of Corollary~{\upshape\ref{c8}}]
    By Theorem~{\upshape\ref{t2.2}} and Corollary~{\upshape\ref{c7}}, we have
    \begin{align*}
        \textup{CF}[3n^2+7n+3,-2n^2(n+2)(n+3)]=\frac{1}{H(2,0;3)}=\frac{1}{\frac{15H(2,0;2)-1}{18}}=\frac{288}{90G-31}.
    \end{align*}
\end{proof}
\begin{proof}[Proof of Corollary~{\upshape\ref{c9}}]
    By Theorem~{\upshape\ref{t2.2}}, we have
    \begin{align*}
        \textup{CF}[3n^2+9n+7,-2(n+1)^3(n+\gamma)]&=\frac{-2\gamma}{\textup{CF}[3n^2+3n+1,-2n^3(n+\gamma-1)]-1}\\
        &=\frac{-2\gamma}{\frac{1}{H(0,0;\gamma-1)}-1}\\
        &=\frac{-2\gamma}{\frac{(2\gamma-1)^2}{2\gamma(2\gamma-1)H(0,0;\gamma)+1}-1}\\
        &=\frac{2\gamma(2\gamma-1)H(0,0;\gamma)+1}{(2\gamma-1)H(0,0;\gamma)-2\gamma+2}.
    \end{align*}
    By putting $\gamma\rightarrow0$, from Corollary~{\upshape\ref{c1}}, we get
    $$\textup{CF}[3n^2+9n+7,-2n(n+1)^3]=\frac{1}{2-2G}.$$
\end{proof}
\begin{proof}[Proof of Corollary~{\upshape\ref{c10}}]
    By Theorem~{\upshape\ref{t2.2}} and Corollary~{\upshape\ref{c5}}, we have
    \begin{align*}
        \textup{CF}[3n^2+11n+5,-2n^3(n+4)]=\frac{1}{H(4,0;0)}=\frac{1}{\frac{9H(2,0;0)-1}{12}}=\frac{24}{18G-11}.
    \end{align*}
\end{proof}
\begin{proof}[Proof of Corollary~{\upshape\ref{c11}}]
    By Theorem~{\upshape\ref{t2.2}} and Corollary~{\upshape\ref{c10}}, we have
    \begin{align*}
        \textup{CF}[3n^2+11n+5,-2n^2(n+1)(n+4)]=\frac{1}{H(4,0;1)}=\frac{1}{\frac{-3H(4,0;0)-1}{-6}}=\frac{16}{6G-1}.
    \end{align*}
\end{proof}
\begin{proof}[Proof of Corollary~{\upshape\ref{c12}}]
    By Theorem~{\upshape\ref{t2.2}} and Corollary~{\upshape\ref{c11}}, we have
    \begin{align*}
        \textup{CF}[3n^2+11n+5,-2n^2(n+2)(n+4)]=\frac{1}{H(4,0;2)}=\frac{1}{\frac{-3H(4,0;1)-1}{-4}}=\frac{64}{18G-13}.
    \end{align*}
\end{proof}
\begin{proof}[Proof of Corollary~{\upshape\ref{c13}}]
    By Theorem~{\upshape\ref{t2.2}} and Corollary~{\upshape\ref{c5}}, we have
    \begin{align*}
        \textup{CF}[3n^2+11n+9,-2n^2(n+2)^2]=\frac{1}{H(2,2;0)}=\frac{1}{\frac{3H(2,0;0)-1}{2}}=\frac{4}{6G-5}.
    \end{align*}
\end{proof}
\begin{proof}[Proof of Corollary~{\upshape\ref{c14}}]
    By Theorem~{\upshape\ref{t2.2}} and Corollary~{\upshape\ref{c13}}, we have
    \begin{align*}
        \textup{CF}[3n^2+11n+9,-2n(n+1)(n+2)^2]=\frac{1}{H(2,2;1)}=\frac{1}{\frac{H(2,2;0)-1}{-6}}=\frac{8}{3-2G}.
    \end{align*}
\end{proof}
\begin{proof}[Proof of Corollary~{\upshape\ref{c15}}]
    By Theorem~{\upshape\ref{t2.2}} and Corollary~{\upshape\ref{c14}}, we have
    \begin{align*}
        \textup{CF}[3n^2+11n+9,-2n(n+2)^3]=\frac{1}{H(2,2;2)}=\frac{1}{\frac{H(2,2;1)-1}{-4}}=\frac{32}{2G+5}.
    \end{align*}
\end{proof}
\begin{proof}[Proof of Corollary~{\upshape\ref{c16}}]
    By Theorem~{\upshape\ref{t2.2}} and Corollary~{\upshape\ref{c15}}, we have
    \begin{align*}
        \textup{CF}[3n^2+11n+9,-2n(n+2)^2(n+3)]=\frac{1}{H(2,2;3)}=\frac{1}{\frac{9H(2,2;2)-1}{6}}=\frac{192}{18G+13}.
    \end{align*}
\end{proof}
\begin{proof}[Proof of Corollary~{\upshape\ref{c17}}]
    By Theorem~{\upshape\ref{t2.2}}, we have
    \begin{align*}
        \textup{CF}[3n^2+13n+13,-2(n+1)^2(n+3)(n+\gamma)]&=\frac{-6\gamma}{\textup{CF}[3n^2+7n+3,-2n^2(n+2)(n+\gamma-1)]-3}\\
        &=\frac{-6\gamma}{\frac{1}{H(2,0;\gamma-1)}-3}\\
        &=\frac{-6\gamma}{\frac{(2\gamma-1)(2\gamma-3)}{2\gamma(2\gamma-3)H(2,0;\gamma)+1}-3}\\
        &=\frac{6\gamma(2\gamma-3)H(2,0;\gamma)+3}{3(2\gamma-3)H(2,0;\gamma)-2\gamma+4}.
    \end{align*}
    By putting $\gamma\rightarrow0$, from Corollary~{\upshape\ref{c5}}, we get
    $$\textup{CF}[3n^2+13n+13,-2n(n+1)^2(n+3)]=\frac{6}{17-18G}.$$
\end{proof}
\begin{proof}[Proof of Corollary~{\upshape\ref{c18}}]
    By Theorem~{\upshape\ref{t2.2}} and Corollary~{\upshape\ref{c13}}, we have
    \begin{align*}
        \textup{CF}[3n^2+15n+15,-2n^2(n+2)(n+4)]=\frac{1}{H(4,2;0)}=\frac{1}{\frac{15H(2,2;0)-1}{12}}=\frac{48}{90G-79}.
    \end{align*}
\end{proof}
\begin{proof}[Proof of Corollary~{\upshape\ref{c19}}]
    By Theorem~{\upshape\ref{t2.2}} and Corollary~{\upshape\ref{c18}}, we have
    \begin{align*}
        \textup{CF}[3n^2+15n+15,-2n(n+1)(n+2)(n+4)]=\frac{1}{H(4,2;1)}=\frac{1}{\frac{3H(4,2;0)-1}{-10}}=\frac{32}{19-18G}.
    \end{align*}
\end{proof}
\begin{proof}[Proof of Corollary~{\upshape\ref{c20}}]
    By Theorem~{\upshape\ref{t2.2}} and Corollary~{\upshape\ref{c19}}, we have
    \begin{align*}
        \textup{CF}[3n^2+15n+15,-2n(n+2)^2(n+4)]=\frac{1}{H(4,2;2)}=\frac{1}{\frac{-H(4,2;1)-1}{-12}}=\frac{128}{17-6G}.
    \end{align*}
\end{proof}
\begin{proof}[Proof of Corollary~{\upshape\ref{c21}}]
    By Theorem~{\upshape\ref{t2.2}} and the proof of Corollary~{\upshape\ref{c9}}, we have
    \begin{align*}
        &\textup{CF}[3n^2+15n+19,-2(n+2)^3(n+\gamma)]\\
        &=\frac{-16\gamma}{\textup{CF}[3n^2+9n+7,-2(n+1)^3(n+\gamma-1)]-7}\\
        &=\frac{-16\gamma}{\frac{2(\gamma-1)(2\gamma-3)H(0,0;\gamma-1)+1}{(2\gamma-3)H(0,0;\gamma-1)-2(\gamma-2)}-7}\\
        &=\frac{16\gamma}{2\gamma-9}\left\{-1+\frac{(2\gamma-3)^2}{(2\gamma-3)(2\gamma-9)H(0,0;\gamma-1)+14\gamma-27}\right\}\\
        &=\frac{16\gamma}{2\gamma-9}\left\{-1+\frac{(2\gamma-3)^2}{(2\gamma-3)(2\gamma-9)\frac{2\gamma(2\gamma-1)H(0,0;\gamma)+1}{(2\gamma-1)^2}+14\gamma-27}\right\}\\
        &=-\frac{16\gamma}{2\gamma-9}+\frac{8}{2\gamma-9}\frac{(2\gamma-1)^2(2\gamma-3)^2}{(2\gamma-1)(2\gamma-3)(2\gamma-9)H(0,0;\gamma)+28\gamma^2-80\gamma+49}.
    \end{align*}
    By putting $\gamma\rightarrow0$, from Corollary~{\upshape\ref{c1}}, we get
    $$\textup{CF}[3n^2+15n+19,-2n(n+2)^3]=\frac{8}{54G-49}.$$
\end{proof}
\begin{proof}[Proof of Corollary~{\upshape\ref{c22}}]
    By Theorem~{\upshape\ref{t2.2}}, we have
    \begin{align*}
        \textup{CF}[3n^2+17n+23,-2(n+1)(n+3)^2(n+\gamma)]&=\frac{-18\gamma}{\textup{CF}[3n^2+11n+9,-2n(n+2)^2(n+\gamma-1)]-9}\\
        &=\frac{-18\gamma}{\frac{1}{H(2,2;\gamma-1)}-9}\\
        &=\frac{-18\gamma}{\frac{(2\gamma-3)^2}{2\gamma(2\gamma-5)H(2,2;\gamma)+1}-9}\\
        &=\frac{18\gamma(2\gamma-5)H(2,2;\gamma)+9}{9(2\gamma-5)H(2,2;\gamma)-2\gamma+6}.
    \end{align*}
    By putting $\gamma\rightarrow0$, from Corollary~{\upshape\ref{c13}}, we get
    $$\textup{CF}[3n^2+17n+23,-2n(n+1)(n+3)^2]=\frac{12}{83-90G}.$$
\end{proof}
There are some cases where different continued fractions converges to the same value like Corollaries~{\upshape\ref{c2}} and~{\upshape\ref{c5}} This fact is generalized as follows.
\begin{thm}\label{t2.3}
    For arbitrary non-negative integers $p$ and $q$, we have
    $$H(2p,0;q)=H(2q,0;p).$$
\end{thm}
\begin{proof}
    Assume that $H_{p,q}\coloneqq H(2p,0;q)$. By Theorem~{\upshape\ref{t2.2}}, $H_{p,q}$ satisfies
    \begin{align*}
        \left\{\begin{array}{l}
2p(2p-2q-1)H_{p,q}-(2p-1)(2p-2q-1)H_{p-1,q}+1=0, \\
2q(2q-2p-1)H_{p,q}-(2q-1)(2q-2p-1)H_{p,q-1}+1=0.
\end{array}
\right.
    \end{align*}
    Since this recurrence formula is symmetric with respect to $p$ and $q$, $H_{p,q}$ is symmetric with respect to $p$ and $q$. This completes the proof.
\end{proof}

The following are corollaries involving $\pi^2$. Corollary~{\upshape\ref{c23}} has already been proved by \cite[Theorem 3]{RMProof} in a different way, which is introduced  more generally in section~{\upshape\ref{sec3}}.
\begin{cor}\label{c23}
    $$\frac{8}{\pi^2}=\textup{CF}[3n^2+3n+1,-n^3(2n-1)]$$
\end{cor}
\begin{cor}\label{c24}
$$\frac{16}{4+\pi^2}=\textup{CF}[3n^2+3n+1,-n^3(2n-3)]$$
\end{cor}
\begin{cor}\label{c25}
$$\frac{16}{-4+\pi^2}=\textup{CF}[3n^2+7n+3,-n^2(n+2)(2n-1)]$$
\end{cor}
\begin{cor}\label{c26}
$$\frac{32}{\pi^2}=\textup{CF}[3n^2+7n+3,-n^2(n+2)(2n-3)]$$
\end{cor}
\begin{cor}\label{c27}
$$\frac{16}{-8+\pi^2}=\textup{CF}[3n^2+11n+9,-n(n+2)^2(2n-1)]$$
\end{cor}
\begin{cor}\label{c28}
$$\frac{16}{12-\pi^2}=\textup{CF}[3n^2+11n+9,-n(n+2)^2(2n+1)]$$
\end{cor}
\begin{cor}\label{c29}
$$\frac{32}{32-3\pi^2}=\textup{CF}[3n^2+15n+15,-n(n+2)(n+4)(2n+1)]$$
\end{cor}
\begin{cor}\label{c30}
$$\frac{16+3\pi^2}{16-\pi^2}=\textup{CF}[3n^2+9n+7,-(n+1)^3(2n-3)]$$
\end{cor}
\begin{proof}[Proof of Corollary~{\upshape\ref{c23}}]
    By substituting $\alpha=\beta=0$ and $\gamma=-\frac{1}{2}$ for Theorem~{\upshape\ref{t2.1}}, we have
    \begin{align*}
        \textup{CF}[3n^2+3n+1,-n^3(2n-1)]&=\frac{1}{H\left(0,0;-\frac{1}{2}\right)}=4\left\{\sum_{n=0}^\infty\frac{\left(\frac{1}{2}\right)_n\left(\frac{1}{2}\right)_n}{\bigl(\frac{1}{2}\bigr)_{n+1}\bigl(\frac{1}{2}\bigr)_{n+1}}\right\}^{-1}\\
        &=\left\{\sum_{n=0}^\infty\frac{1}{(2n+1)^2}\right\}^{-1}\\
        &=\frac{8}{\pi^2}.
    \end{align*}
\end{proof}
\begin{proof}[Proof of Corollary~{\upshape\ref{c24}}]
    By Theorem~{\upshape\ref{t2.2}} and Corollary~{\upshape\ref{c23}}, we have
    $$\textup{CF}[3n^2+3n+1,-n^3(2n-3)]=\frac{1}{H\left(0,0;-\frac{3}{2}\right)}=\frac{1}{\frac{2H\left(0,0;-\frac{1}{2}\right)+1}{4}}=\frac{16}{4+\pi^2}.$$
\end{proof}
\begin{proof}[Proof of Corollary~{\upshape\ref{c25}}]
    By Theorem~{\upshape\ref{t2.2}} and Corollary~{\upshape\ref{c23}}, we have
    $$\textup{CF}[3n^2+7n+3,-n^2(n+2)(2n-1)]=\frac{1}{H\left(2,0;-\frac{1}{2}\right)}=\frac{1}{\frac{2H\left(0,0;-\frac{1}{2}\right)-1}{4}}=\frac{16}{-4+\pi^2}.$$
\end{proof}
\begin{proof}[Proof of Corollary~{\upshape\ref{c26}}]
    By Theorem~{\upshape\ref{t2.2}} and Corollary~{\upshape\ref{c25}}, we have
    $$\textup{CF}[3n^2+7n+3,-n^2(n+2)(2n-3)]=\frac{1}{H\left(2,0;-\frac{3}{2}\right)}=\frac{1}{\frac{4H\left(2,0;-\frac{1}{2}\right)+1}{8}}=\frac{32}{\pi^2}.$$
\end{proof}
\begin{proof}[Proof of Corollary~{\upshape\ref{c27}}]
    By Theorem~{\upshape\ref{t2.2}} and Corollary~{\upshape\ref{c25}}, we have
    $$\textup{CF}[3n^2+11n+9,-n(n+2)^2(2n-1)]=\frac{1}{H\left(2,2;-\frac{1}{2}\right)}=\frac{1}{\frac{4H\left(2,0;-\frac{1}{2}\right)-1}{4}}=\frac{16}{-8+\pi^2}.$$
\end{proof}
\begin{proof}[Proof of Corollary~{\upshape\ref{c28}}]
    By Theorem~{\upshape\ref{t2.2}} and Corollary~{\upshape\ref{c27}}, we have
    $$\textup{CF}[3n^2+11n+9,-n(n+2)^2(2n+1)]=\frac{1}{H\left(2,2;\frac{1}{2}\right)}=\frac{1}{\frac{4H\left(2,2;-\frac{1}{2}\right)-1}{-4}}=\frac{16}{12-\pi^2}.$$
\end{proof}
\begin{proof}[Proof of Corollary~{\upshape\ref{c29}}]
    By Theorem~{\upshape\ref{t2.2}} and Corollary~{\upshape\ref{c28}}, we have
    $$\textup{CF}[3n^2+15n+15,-n(n+2)(n+4)(2n+1)]=\frac{1}{H\left(4,2;\frac{1}{2}\right)}=\frac{1}{\frac{12H\left(2,2;\frac{1}{2}\right)-1}{8}}=\frac{32}{32-3\pi^2}.$$
\end{proof}
\begin{proof}[Proof of Corollary~{\upshape\ref{c30}}]
    By Theorem~{\upshape\ref{t2.2}} and Corollary~{\upshape\ref{c24}}, we have
    $$\textup{CF}[3n^2+3n+1,-n^3(2n-5)]=\frac{1}{H\left(0,0;-\frac{5}{2}\right)}=\frac{1}{\frac{12H\left(0,0;-\frac{3}{2}\right)+1}{16}}=\frac{64}{16+3\pi^2}.$$
    Thus,
    $$\textup{CF}[3n^2+9n+7,-(n+1)^3(2n-3)]=\frac{3}{\textup{CF}[3n^2+3n+1,-n^3(2n-5)]-1}=\frac{16+3\pi^2}{16-\pi^2}.$$
\end{proof}

The following is a corollary involving $\log(2)$.
\begin{cor}\label{c31}
$$\frac{1}{1-\log(2)}=\textup{CF}[3n^2+7n+4,-2n^2(n+1)^2]$$
\end{cor}
\begin{proof}
    By substituting $\alpha=\beta=1$ and $\gamma=0$ for Theorem~{\upshape\ref{t2.1}}, we have
    \begin{align*}
        \textup{CF}[3n^2+7n+4,-2n^2(n+1)^2]&=4\left\{\sum_{n=0}^\infty\frac{\left(\frac{1}{2}\right)_n(1)_n}{(1)_{n+1}(1)_{n+1}}\right\}^{-1}\\
        &=\left\{\sum_{n=1}^\infty\frac{(2n)!}{4^n(n!)^2}\frac{1}{2n(2n-1)}\right\}^{-1}.
    \end{align*}
    Here, from
    $$\vartheta_z \log\left(\frac{1+\sqrt{1+z^2}}{2}\right) = 1- \frac{1}{\sqrt{1+z^2}},$$
    we get
    $$\log\left(1+\sqrt{1+z^2}\right)=\log(2)-\sum_{n=1}^\infty\frac{(-1)^n(2n)!}{4^n(n!)^2}\frac{z^{2n}}{2n},\qquad|z|\leq1.$$
    Substituting $z=i$, we have
    $$\sum_{n=1}^\infty\frac{(2n)!}{4^n(n!)^2}\frac{1}{2n}=\log(2).$$
    In addition,
    \begin{align*}
        \sum_{n=1}^\infty\frac{(2n)!}{4^n(n!)^2}\frac{1}{2n-1}&=\sum_{n=0}^\infty\frac{(2n)!}{4^n(n!)^2}\frac{1}{2(n+1)}\\
        &=\sum_{n=0}^\infty\frac{(2n)!}{4^n(n!)^2}-\frac{(2n+2)!}{4^{n+1}((n+1)!)^2}\\
        &=1.
    \end{align*}
    Therefore,
    \begin{align*}
        \textup{CF}[3n^2+7n+4,-2n^2(n+1)^2]&=\left\{\sum_{n=1}^\infty\frac{(2n)!}{4^n(n!)^2}\frac{1}{2n-1}-\sum_{n=1}^\infty\frac{(2n)!}{4^n(n!)^2}\frac{1}{2n}\right\}^{-1}\\
        &=\frac{1}{1-\log(2)}.
    \end{align*}
\end{proof}

\section{Proofs of the conjectures for $\pi^2$, $\log(2)$ and special values of the zeta function by Petkov\v{s}ek's algorithm}\label{sec3}
With the special solution of (\upshape\ref{e1.2}), the continued fraction can be expressed as the following lemma. 

\begin{lem}\label{l3.1}Assume that a sequence $y_n$ satisfies (\upshape\ref{e1.2}). If $y_n\neq0$ for all $n\geq0$,
$$\textup{CF}[a_n,b_n]=\left(a_0-\frac{y_1}{y_0}\right)+\frac{1}{{y_0}^2}\left\{\sum_{k=0}^\infty\frac{(-1)^k}{y_k y_{k+1}}\prod_{i=1}^kb_i\right\}^{-1}.$$
\end{lem}
It is already known how to construct a general solution from a special solution in a second-order linear difference equation. Thus Lemma~{\upshape\ref{l3.1}} is easily proved by applying Lemma~{\upshape\ref{l1.1}}.

\begin{proof}Multiplying the both sides of (\upshape\ref{e1.2}) and the difference equation of $A_n$ in (\upshape\ref{e1.1}) by $A_n$ and $y_n$ respectively, we have 
\begin{eqnarray*}
\left\{\begin{array}{l}
\displaystyle A_ny_{n+1}=a_nA_ny_n+b_nA_ny_{n-1}, \\
\displaystyle A_{n+1}y_n=a_nA_ny_n+b_nA_{n-1}y_n.
\end{array}
\right.
\end{eqnarray*}
Subtracting the first equation by the second equation on both sides, we get
\begin{eqnarray*}
(A_{n+1}y_n-A_ny_{n+1})&=&-b_n(A_ny_{n-1}-A_{n-1}y_n)\\
&=&(A_1y_0-A_0y_1)\prod_{i=1}^n(-b_i).
\end{eqnarray*}
Thus,
$$\frac{A_{n+1}}{y_{n+1}}-\frac{A_n}{y_n}=(a_0y_0-y_1)\frac{\prod_{i=1}^n(-b_i)}{y_ny_{n+1}}.$$
Therefore we get
$$ A_n=y_n\left\{\frac{1}{y_0}+(a_0y_0-y_1)\sum_{k=0}^{n-1}\frac{\prod_{i=1}^k(-b_i)}{y_ky_{k+1}}\right\}.$$
By a similar argument, we have
$$B_n=y_0y_n\sum_{k=0}^{n-1}\frac{\prod_{i=1}^k(-b_i)}{y_ky_{k+1}}.$$
It follows from Lemma~{\upshape\ref{l1.1}} that
$$\text{CF}[a_n,b_n]=\lim_{n\to\infty}\frac{A_n}{B_n}=\left(a_0-\frac{y_1}{y_0}\right)+\frac{1}{{y_0}^2}\left\{\sum_{k=0}^\infty\frac{(-1)^k}{y_k y_{k+1}}\prod_{i=1}^kb_i\right\}^{-1}.$$
\end{proof}

A special solution of (\upshape\ref{e1.2}) can sometimes be easily found by applying Petkov\v{s}ek's algorithm, which uses the following specific representation of rational functions.
\begin{lem}\label{3.2}\textup{\cite[Theorem 5.3.1]{A=B}}
Let $f(n)$ be a non-zero rational function. Then there exist polynomials $p,q,r$ such that $p,r$ are monic and
$$f(n)=\frac{p(n)}{p(n-1)}\frac{q(n)}{r(n)},$$
where
\begin{itemize}
    \item[\textup{1.}] $\gcd(q(n),r(n+j))=1\qquad$ for all non-negative integer $j,$
    \item[\textup{2.}] $\gcd(p(n-1),q(n))=1$,
    \item[\textup{3.}] $\gcd(p(n),r(n))=1.$
\end{itemize}
\end{lem}
By using Petkov\v{s}ek's algorithm, the following theorems can be proved.

\begin{thm}\label{t3.1}
$$\frac{18}{-8+\pi^2}=\textup{CF}[5n^2+14n+10,-2n^3(2n+3)]$$
\end{thm}
\begin{thm}\label{t3.2}
$$\frac{1}{-1+2\log(2)}=\textup{CF}[3n+3,-2n^2]$$
\end{thm}
\begin{thm}\label{t3.3}
$$-\frac{1}{\zeta(4)+4\zeta(2)-8}=\textup{CF}[n^4+(n+1)^4+2\left\{n^2+(n+1)^2\right\},-n^8]$$
\end{thm}
\begin{thm}\label{t3.4}
$$\frac{2}{2\zeta(5)+6\zeta(3)-9}=\textup{CF}[n^5+(n+1)^5+6\left\{n^3+(n+1)^3\right\},-n^{10}]$$
\end{thm}

\begin{thm}\label{t3.5}
$$\frac{2}{2\zeta(5)-2\zeta(3)-1}=\textup{CF}[n^5+(n+1)^5+6\left\{n^3+(n+1)^3\right\}-4(2n+1),-n^{10}]$$
\end{thm}

\begin{thm}\label{t3.6}
$$\frac{64}{64\zeta(5)+176\zeta(3)-273}=\textup{CF}[n^5+(n+1)^5+16\left\{n^3+(n+1)^3\right\}-4(2n+1),-n^{10}]$$
\end{thm}

\begin{thm}\label{t3.7}
$$\frac{1}{\zeta(7)-4\zeta(3)+4}=\textup{CF}[n^7+(n+1)^7+8\left\{n^5+(n+1)^5\right\}-8\left\{n^3+(n+1)^3\right\}+4(2n+1),-n^{14}]$$
\end{thm}

Derivation of the solution of (\upshape\ref{e1.2}) using Petkov\v{s}ek's algorithm is omitted because it is easy to verify that a given sequence $y_n$ satisfies (\upshape\ref{e1.2}).

\begin{proof}[Proof of Theorem~{\upshape\ref{t3.1}}]
Substitute $a_n=5n^2+14n+10$ and $b_n=-2n^3(2n+3)$ for (\upshape\ref{e1.2}). Then it has the following solution:
$$y_n=2^nn!(2n+3)!!.$$
Thus,
\begin{align*}
    \frac{(-1)^k}{y_ky_{k+1}}\prod_{i=1}^kb_i&=\frac{(-1)^k}{2^kk!(2k+3)!!2^{k+1}(k+1)!(2k+5)!!}\cdot(-2)^k(k!)^3\frac{(2k+3)!!}{3}\\
    &=\frac{1}{6}\frac{1}{(k+1)(2k+1)(2k+3)(2k+5)}\frac{(k!)^2}{(2k)!}.
\end{align*}
Thus,
\begin{align*}
    \sum_{k=0}^\infty\frac{(-1)^k}{y_ky_{k+1}}\prod_{i=1}^kb_i&=\frac{1}{6}\sum_{k=0}^\infty\frac{(k!)^2}{(2k)!}\frac{1}{(k+1)(2k+1)(2k+3)(2k+5)}\\
    &=\frac{1}{6}\sum_{k=0}^\infty\frac{(k!)^2}{(2k)!}\left(-\frac{1}{3}\cdot\frac{1}{k+1}+\frac{1}{4}\cdot\frac{1}{2k+1}+\frac{1}{2}\cdot\frac{1}{2k+3}-\frac{1}{12}\cdot\frac{1}{2k+5}\right)\\
    &=-\frac{1}{18}\left(\frac{4\pi\sqrt{3}}{9}-\frac{\pi^2}{9}\right)+\frac{1}{24}\cdot\frac{2\pi\sqrt{3}}{9}+\frac{1}{12}\left(\frac{14\pi\sqrt{3}}{9}-8\right)-\frac{1}{72}\left(\frac{74\pi\sqrt{3}}{9}-\frac{400}{9}\right)\\&=\frac{\pi^2-8}{162},
\end{align*}
where we use ~\cite[Eq.3.43,3.44,3.45,3.79]{GS}. Therefore from Lemma~{\upshape\ref{l3.1}}, we see that
$$\textup{CF}[5n^2+14n+10,-2n^3(2n+3)]=\frac{1}{9}\cdot\frac{162}{-8+\pi^2}=\frac{18}{-8+\pi^2}.$$
\end{proof}
\begin{proof}[Proof of Theorem~{\upshape\ref{t3.2}}]
Substitute $a_n=3n+3$ and $b_n=-2n^2$ for (\upshape\ref{e1.2}). Then it has the following solution:
$$y_n=2^n(n+2)n!.$$
Thus,
$$\frac{(-1)^k}{y_ky_{k+1}}\prod_{i=1}^kb_i=\frac{(-1)^k}{2^k(k+2)k!\cdot2^{k+1}(k+3)(k+1)!}\cdot (-2)^k(k!)^2=\frac{1}{2}\frac{2^{-k}}{(k+1)(k+2)(k+3)}.$$
Therefore by Lemma~{\upshape\ref{l3.1}}, we get
\begin{align*}
\text{CF}[3n+3,-2n^2]&=\frac{1}{4}\left\{\frac{1}{2}\sum_{k=0}^\infty2^{-k}\left(\frac{1}{2}\cdot\frac{1}{k+1}-\frac{1}{k+2}+\frac{1}{2}\cdot\frac{1}{k+3}\right)\right\}^{-1}\\
&=\frac{1}{4}\left\{\frac{1}{2}\sum_{k=1}^\infty\frac{2^{-k}}{k}-2\left(\sum_{k=1}^\infty\frac{2^{-k}}{k}-\frac{1}{2}\right)+2\left(\sum_{k=1}^\infty\frac{2^{-k}}{k}-\frac{5}{8}\right)\right\}^{-1}\\
&=\frac{1}{-1+2\log(2)}.
\end{align*}
\end{proof}
\begin{proof}[Proof of Theorem~{\upshape\ref{t3.3}}]
Substitute $a_n=n^4+(n+1)^4+2\{n^2+(n+1)^2\}$ and $b_n=-n^8$ for (\upshape\ref{e1.2}). Then it has the following solution:
$$y_n=(2n+1)(n!)^4.$$
Therefore by Lemma~{\upshape\ref{l3.1}}, we get
\begin{align*}
\text{CF}[n^4+(n+1)^4+2\left\{n^2+(n+1)^2\right\},-n^8]&=\left\{\sum_{k=1}^\infty\frac{1}{k^4(2k+1)(2k-1)}\right\}^{-1}\\
&=\left\{\sum_{k=1}^\infty\left(-\frac{1}{k^4}-\frac{4}{k^2}-\frac{8}{2k+1}-\frac{8}{2k-1}\right)\right\}^{-1}\\
&=-\frac{1}{\zeta(4)+4\zeta(2)-8}.
\end{align*}
\end{proof}
\begin{proof}[Proof of Theorem~{\upshape\ref{t3.4}}]
Substitute $a_n=n^5+(n+1)^5+6\left\{n^3+(n+1)^3\right\}$ and $b_n=-n^{10}$ for (\upshape\ref{e1.2}). Then it has the following solution:
$$y_n=(3n^2+3n+1)(n!)^5.$$
Therefore by Lemma~{\upshape\ref{l3.1}}, we get
\begin{align*}
&\text{CF}[n^5+(n+1)^5+6\left\{n^3+(n+1)^3\right\},-n^{10}]\\
&=\left\{\sum_{k=1}^\infty\frac{1}{k^5(3k^2+3k+1)(3k^2-k+1)}\right\}^{-1}\\
&=\left\{\sum_{k=1}^\infty\left(\frac{1}{k^5}+\frac{3}{k^3}+\frac{9}{2(3k^2+3k+1)}-\frac{9}{2(3k^2-3k+1)}\right)\right\}^{-1}\\
&=\frac{2}{2\zeta(5)+6\zeta(3)-9}.
\end{align*}
\end{proof}
\begin{proof}[Proof of Theorem~{\upshape\ref{t3.5}}]
Substitute $a_n=n^5+(n+1)^5+6\left\{n^3+(n+1)^3\right\}-4(2n+1)$ and $b_n=-n^{10}$ for (\upshape\ref{e1.2}). Then it has the following solution:
$$y_n=(n^2+n+1)(n!)^5.$$
Therefore by Lemma~{\upshape\ref{l3.1}}, we get
\begin{align*}
&\text{CF}[n^5+(n+1)^5+6\left\{n^3+(n+1)^3\right\}-4(2n+1),-n^{10}]\\
&=\left\{\sum_{k=1}^\infty\frac{1}{k^5(k^2+k+1)(k^2-k+1)}\right\}^{-1}\\
&=\left\{\sum_{k=1}^\infty\left(\frac{1}{k^5}-\frac{1}{k^3}-\frac{1}{2(k^2+k+1)}+\frac{1}{2(k^2-k+1)}\right)\right\}^{-1}\\
&=\frac{2}{2\zeta(5)-2\zeta(3)-1}.
\end{align*}
\end{proof}
\begin{proof}[Proof of Theorem~{\upshape\ref{t3.6}}]
Substitute $a_n=n^5+(n+1)^5+16\left\{n^3+(n+1)^3\right\}-4(2n+1)$ and $b_n=-n^{10}$ for (\upshape\ref{e1.2}). Then it has the following solution:
$$y_n=(5n^4+10n^3+19n^2+14n+4)(n!)^5. $$
Therefore by Lemma~{\upshape\ref{l3.1}}, we get
\begin{eqnarray*}
&&\text{CF}[n^5+(n+1)^5+16\left\{n^3+(n+1)^3\right\}-4(2n+1),-n^{10}]\\
&&=\frac{1}{16}\left\{\sum_{k=1}^\infty\frac{1}{k^5(5k^4+10k^3+19k^2+14k+4)(5k^4-10k^3+19k^2-14k+4)}\right\}^{-1}\\
&&=\left\{\sum_{k=1}^\infty\left(\frac{1}{k^5}+\frac{11}{4k^3}-\frac{110k^2-110k+273}{16(5k^4-10k^3+19k^2-14k+4)}+\frac{110k^2+110k+273}{16(5k^4+10k^3+19k^2+14k+4)}\right)\right\}^{-1}\\
&&=\frac{64}{64\zeta(5)+176\zeta(3)-273}.
\end{eqnarray*}
\end{proof}
\begin{proof}[Proof of Theorem~{\upshape\ref{t3.7}}]
Substitute $a_n=n^7+(n+1)^7+8\left\{n^5+(n+1)^5\right\}-8\left\{n^3+(n+1)^3\right\}+4(2n+1)$ and $b_n=-n^{14}$ for (\upshape\ref{e1.2}). Then it has the following solution:
$$y_n=(2n^2+2n+1)(n!)^7.$$
Therefore by Lemma~{\upshape\ref{l3.1}}, we get
\begin{align*}
&\text{CF}[n^7+(n+1)^7+8\left\{n^5+(n+1)^5\right\}-8\left\{n^3+(n+1)^3\right\}+4(2n+1),-n^{14}]\\
&=\left\{\sum_{k=1}^\infty\frac{1}{k^7(2k^2+2k+1)(2k^2-2k+1)}\right\}^{-1}\\
&=\left\{\sum_{k=1}^\infty\left(\frac{1}{k^7}-\frac{4}{k^3}-\frac{4}{2k^2+2k+1}+\frac{4}{2k^2-2k+1}\right)\right\}^{-1}\\
&=\frac{1}{\zeta(7)-4\zeta(3)+4}.
\end{align*}
\end{proof}

Corollaries 2.23--31 can also be proved in the same way.
However, the way of finding the value of the series in the proofs of Corollaries~{\upshape\ref{c24}},~{\upshape\ref{c26}},~{\upshape\ref{c27}} and~{\upshape\ref{c30}} is more complicated than the others\footnote{I express my sincere thanks to Taiki Watanabe for proving the four corollaries in this way.}.

For example, another proof of Corollary~{\upshape\ref{c24}} is given below.

\begin{proof}[Proof of Corollary~{\upshape\ref{c24}}]
Substitute $a_n=3n^2+3n+1$ and $b_n=-n^3(2n-3)$ for (\upshape\ref{e1.2}). Then it has the following solution:
$$y_n=(n^2+3n-2)n!(2n-3)!!.$$
Here assume that $(-1)!!=1,(-3)!!=-1$. Thus we have
\begin{align*}
\frac{(-1)^k}{y_ky_{k+1}}\prod_{i=1}^kb_i
&=\frac{(-1)^k}{(k^2+3k-2)k!(2k-3)!!(k^2+5k+2)(k+1)!(2k-1)!!}\cdot (-1)^{k+1}(k!)^3(2k-3)!!\\
&=-\frac{1}{(k+1)(k^2+3k-2)(k^2+5k+2)}\frac{(k!)^22^k}{(2k)!}.
\end{align*}
Thus, we get
\begin{align*}
&\sum_{k=0}^\infty\frac{(-1)^k}{y_k y_{k+1}}\prod_{i=1}^k b_i\\
&=-\sum_{k=0}^\infty\frac{(k!)^22^k}{(2k)!}\frac{1}{(k+1)(k^2+3k-2)(k^2+5k+2)}\\
&=-\sum_{k=0}^\infty\frac{(k!)^2}{(2k)!}\frac{2^k}{(k+1)(2k+4)}\left(\frac{1}{k^2+3k-2}-\frac{1}{k^2+5k+2}\right)\\
&=\frac{1}{8}-\sum_{k=0}^\infty\frac{((k+1)!)^2}{(2k+2)!}\frac{2^k}{(k+2)(k+3)(k^2+5k+2)}+\sum_{k=0}^\infty\frac{(k!)^2}{(2k)!}\frac{2^{k-1}}{(k+1)(k+2)(k^2+5k+2)}\\
&=\frac{1}{8}-\sum_{k=0}^\infty \frac{(k!)^2}{(2k)!}\frac{2^{k-1}}{(k+2)(k^2+5k+2)}\left(\frac{k+1}{(k+3)(2k+1)}-\frac{1}{k+1}\right)\\
&=\frac{1}{8}+\frac{1}{2}\sum_{k=0}^\infty \frac{(k!)^2}{(2k)!}\frac{2^k}{(k+1)(k+2)(k+3)(2k+1)}\\
&=\frac{1}{8}+\frac{1}{2}\sum_{k=0}^\infty \frac{(k!)^22^k}{(2k)!}\left(-\frac{1}{2}\cdot \frac{1}{k+1}+\frac{1}{3}\cdot \frac{1}{k+2}-\frac{1}{10}\cdot \frac{1}{k+3}+\frac{8}{15}\cdot \frac{1}{2k+1}\right)\\
&=\frac{1}{8}+\frac{1}{2}\left\{-\frac{1}{2}\left(\pi-\frac{\pi^2}{8}\right)+\frac{1}{3}\left(\frac{5}{2}\pi-\frac{3}{8}\pi^2-3\right)-\frac{1}{10}\left(6\pi-\frac{15}{16}\pi^2-\frac{35}{4}\right)+\frac{8}{15}\cdot \frac{\pi}{2}\right\}\\
&=\frac{4+\pi^2}{64},
\end{align*}
where we use \cite[Eq. 3.58, 3.94, 3.95 and 3.96]{GS}. Therefore from Lemma~{\upshape\ref{l3.1}}, we see that
\begin{eqnarray*}
\text{CF}[3n^2+3n+1,-n^3(2n-3)]=\frac{1}{4}\cdot \frac{64}{4+\pi^2}=\frac{16}{4+\pi^2}.
\end{eqnarray*}
\end{proof}

\bmhead{Acknowledgments}
I would like to thank Professors Hiroyuki Ochiai and Hideyuki Ishi for their great help in writing this paper.






\end{document}